\documentclass[a4paper,12pt,twoside]{article}
\usepackage{amsmath,amsfonts,amssymb}

\usepackage{amsbsy}
\usepackage{amssymb}
\usepackage{amsfonts}
\usepackage{amsmath}
\usepackage{amsopn}
\usepackage{amsthm}
\usepackage[english]{babel}
\usepackage[T1]{fontenc}
\usepackage{graphicx}

\newtheorem{teo}{Theorem}
\newtheorem{lemma}{Lemma}
\newtheorem{prop}{Proposition}
\newtheorem{cor}{Corollary}

\theoremstyle{definition}
\newtheorem{remark}{Remark}

\theoremstyle{definition}

\theoremstyle{definition}
\newtheorem{example}{Example}

\theoremstyle{remark}

\newcommand{\gracias}{\noindent\textbf{Acknowledgement.}\ }

\DeclareMathOperator{\nr}{\mathbb{N}}

\DeclareMathOperator{\re}{\mathbb{R}}

\DeclareMathOperator{\er}{\mathbf{E}}
\DeclareMathOperator{\pr}{\mathbf{P}}

\DeclareMathOperator{\p}{\mbox{\rm I\hspace{-0.02in}P}}
\DeclareMathOperator{\e}{\mbox{\rm I\hspace{-0.02in}E}}

\title{On the asymptotic behaviour of increasing self-similar Markov processes}
\author{Maria Emilia Caballero
A.\thanks{Instituto de Matem\'aticas, UNAM. E-mail
marie@matem.unam.mx} \ V{\'\i }ctor Rivero\thanks{Centro de
Investigaci\'on en Matem\'aticas (CIMAT A.C.) E-mail:
riverovm@gmail.com}}
\date{\today}

\begin{document}
\maketitle
\begin{abstract}
It has been proved by Bertoin and Caballero \cite{BC2002} that a $1/\alpha$-increasing
self-similar Markov process $X$ is such that $t^{-1/\alpha}X(t)$
converges weakly, as $t\to\infty,$ to a degenerated r.v. whenever
the subordinator associated to it via Lamperti's transformation has
infinite mean. Here we prove that $\log(X(t)/t^{1/\alpha})/\log(t)$ converges in law to a
non-degenerated r.v. if and only if the associated subordinator has Laplace exponent that varies regularly at $0.$ 
 Moreover, we show that
$\liminf_{t\to\infty}\log(X(t))/\log(t)=1/\alpha,$ a.s. and provide
an integral test for the upper functions of $\{\log(X(t)), t\geq 0\}.$ Furthermore, results concerning the rate of growth of the random clock appearing in Lamperti's transformation are obtained. In particular, these allow us to establish estimates for the left tail of some exponential functionals of subordinators. Finally, some of the implications of these results in the theory of self-similar fragmentations are discussed.    
\end{abstract}
\begin{section}{Introduction}
Let $X=\{X(t), t\geq 0\}$ be a positive self-similar Markov process
with c\`adl\`ag and increasing paths, viz. $X$ is a $]0,\infty[$
valued strong Markov process that fulfills the scaling property:
there exists an $\alpha>0$ such that for every $c>0$
$$\left(\{cX(t c^{-\alpha}), t\geq 0\}, \p_x\right)\stackrel{\text{Law}}{=}\left(\{X(t), t\geq
0\}, \p_{cx}\right),\qquad x\in]0,\infty[,$$ where $\p_y$ denotes
the law of the process $X$ with starting point $y>0.$ We will say that
$X$ is an increasing $1/\alpha$-pssMp. Examples of this class of processes are: stable subordinators; in the theory of extremes, the extremal process with $Q$-function given by $ax^{-b},$ for $x>0$ and $\infty$ in other case, for some indices $a,b>0$ (for a more precise description of the latter example see \cite{rivero2003} Section 5); in the theory of self-similar fragmentations, the reciprocal of the process of a tagged fragment, see \cite{bertoinFC} Section 3.3.

It is well known that by means of a transformation due to
Lamperti~\cite{lamperti2} any increasing positive self-similar
Markov processes can be transformed into a subordinator and
vice-versa. By a subordinator we mean a c\`adl\`ag real valued
process with independent and stationary increments, that is, a L\'evy process
with increasing paths. To be more precise about Lamperti's
transformation, given an increasing $1/\alpha$-pssMp $X$ we define a
new process $\xi$ by
$$\xi_t=\log\left(\frac{X(\gamma_t)}{X(0)}\right),\qquad t\geq 0,$$ where
$\{\gamma_t, t\geq 0\}$ denotes the inverse of the additive functional
$$\int^{t}_0X^{-\alpha}_s\mathrm{d}s,\quad t\geq 0.$$ The
process $\xi=\{\xi_t, t\geq 0\}$ defined this way is a subordinator
started from $0$, and we denote by $\pr$ its law. Reciprocally,
given a subordinator $\xi$ and $\alpha>0,$ the process constructed
in the following way is an increasing $1/\alpha$-pssMp. For $x>0,$ we
denote by $\p_x$ the law of the process
$$x\exp\{\xi_{\tau(t/x^{\alpha})}\},\qquad t \geq 0,$$ where $\{\tau(t), t\geq
0\}$ is the inverse of the additive functional
\begin{equation}\label{eq:defA}
C_t:=\int^{t}_0\exp\{\alpha\xi_s\}\mathrm{d}s,\qquad t\geq 0.
\end{equation}
So for any $x>0,$ $\p_x,$ is the law of an $1/\alpha$-pssMp started
from $x>0.$ We will refer to any of these transformations as
Lamperti's transformation.

In a recent work Bertoin and Caballero~\cite{BC2002} studied the
problem of existence of entrance laws at $0+$ for increasing pssMp.
In that work Bertoin and Caballero established that if the
subordinator $(\xi,\pr)$ associated to $(X,\p)$ via Lamperti's
transformation has finite mean $m:=\er(\xi_1)<\infty,$ then there
exists a non-degenerated probability measure $\p_{0+}$ on the space
of paths that are right continuous and left limited which is the
limit in the sense of finite dimensional laws of $\p_x$ as $x\to
0+.$ Using the scaling and Markov properties it is easy to see that
the latter result is equivalent to the weak convergence of random
variables
\begin{equation}\label{eq:weakconvergence}
t^{-1/\alpha}X(t)\xrightarrow[t\to\infty]{\text{Law}} Z,
\end{equation}
where $X$ is started at $1$ and $Z$ is a non-degenerated random
variable. The law of $Z$ will be denoted by $\mu,$ and it is the
probability measure defined by
\begin{equation}\label{eq:limitlaw}
\mu(f):=\e_{0+}\left(f(X(1))\right)=\frac{1}{\alpha
m}\er\left(f\left(\left(\frac{1}{I}\right)^{1/\alpha}\right)\frac{1}{I}\right),
\end{equation} for any measurable function $f:\re^+\to\re^+;$ where $I$ is the exponential functional
$$I:=\int^{\infty}_0\exp\{-\alpha\xi_s\}\mathrm{d}s,$$ associated to
the subordinator $\xi,$ see~\cite{BC2002}. The following result
complements that of Bertoin and Caballero.
\begin{prop}\label{prop:1}
Let $\{X(t), t\geq 0\}$ be a positive $1/\alpha$-self-similar Markov
process with increasing paths. Assume that the subordinator $\xi,$
associated to $X$ via Lamperti's transformation has finite mean,
$m=\er(\xi_1)<\infty.$ Then
$$\frac{1}{\log(t)}\int^{t}_0f(s^{-1/\alpha}X(s))\frac{\mathrm{d}s}{s}\xrightarrow[t\to\infty]{}\mu(f),\quad
\p_{0+}-\text{a.s.}$$ for every function $f\in L^{1}(\mu).$
Furthermore,
$$\frac{\log\left(X(t)\right)}{\log(t)}\xrightarrow[t\to\infty]{} 1/\alpha,\qquad \p_1-\text{a.s.}$$
\end{prop}
In fact, the results of the previous proposition are not new, the first assertion can be obtained as a consequence of an ergodic
theorem for self-similar processes due to Cs{\'a}ki and
F{\"o}ldes~\cite{csakifoldes}, and the second assertion has been obtained in \cite{beasymp}. However, we provide a proof of these results for ease of reference. 

In the work~\cite{BC2002} the authors also proved that if the
subordinator $(\xi,\pr)$ has infinite mean then the convergence in
law in (\ref{eq:weakconvergence}) still holds but $Z$ is a
degenerated random variable equal to $\infty$ a.s. The main purpose
of this work is study in this setting the rate at which
$t^{-1/\alpha}X(t)$ tends to infinity as the time growths.

Observe that the asymptotic behaviour of $(X,\p)$ at large times is
closely related to the large jumps of it, because it is so for the
subordinator $(\xi,\pr).$ So, for our ends it will be important to
have some information about the large jumps of $(\xi,\pr)$ or
equivalently on those of $(X,\p).$ Such information will be provided
by the following assumption. Let $\phi:\re^+\to\re^+$ be the Laplace
exponent of $(\xi,\pr),$ viz.
$$\phi(\lambda):=-\log\left(\er(e^{-\lambda\xi_1})\right)=d\lambda+\int_{]0,\infty[}(1-e^{-\lambda x})\Pi(\mathrm{d}x),\qquad \lambda\geq 0,$$
where $d\geq 0$ and $\Pi$ is a measure on $]0,\infty[$ such that
$\int(1\wedge x)\Pi(\mathrm{d}x)<\infty,$ they are called the drift
term and L\'evy measure of $\xi,$ respectively. We will assume that
$\phi$ is regularly varying at $0,$ i.e.
$$\lim_{\lambda\to 0}\frac{\phi(c\lambda)}{\phi(\lambda)}=c^{\beta},\qquad c>0,$$
for some $\beta\in[0,1],$ which will be called the index of regular
variation of $\phi.$ In the case where $\beta=0,$ it is said that the function $\phi$ is slowly varying. It is known that $\phi$ is regularly varying at
$0$ with an index $\beta\in]0,1[$ if and only if the right tail of
the L\'evy measure $\Pi$ is regularly varying with index $-\beta$, viz.
\begin{equation}\label{taillevy}\lim_{x\to\infty}\frac{\Pi]cx,\infty[}{\Pi]x,\infty[}=c^{-\beta},\qquad
c>0.\end{equation} Well known examples of subordinators whose Laplace exponent is regularly varying are the stable subordinators  and  the Gamma subordinator. A quite rich but less known class of su\-bor\-di\-na\-tors whose Laplace exponent is regularly varying at $0$ is that one of tempered stable subordinators, see \cite{rosinski2007} for background on tempered stable laws. In this case, the drift term is equal to $0$, and the L\'evy measure $\Pi_{\delta}$ has the form $\Pi_{\delta}(\mathrm{d}x)=x^{-\delta-1}q(x)\mathrm{d}x,$  $x>0,$ where $\delta\in]0,1[$ and $q:\re^+\to\re^+$ is a completely monotone function such that $\int^1_{0}x^{-\delta}q(x)\mathrm{d}x<\infty.$ By L'H\^opital rule's, for $\Pi_{\delta}$ to be such that that the condition (\ref{taillevy}) is satisfied it is necessary and sufficient that $q$ be regularly varying at infinity with index $-\lambda$ and such that $0< \lambda+\delta<1.$ By the theory of completely monotone functions it is known that there exists a measure $\mu,$ over $[0,\infty[$ such that $q$ can be represented as $q(x)=\int^\infty_{0}e^{-xy}\mu(\mathrm{d}y)$ for $x\geq 0.$ Owing to Karamata's Tauberian Theorem (Theorem 1.7.1 in \cite{BGT}) it follows that the condition (\ref{taillevy}) is equivalent to the regular variation at zero of the function $x\mapsto\mu[0,x],$ $x>0,$ with index $-\lambda$ for some $0\leq \lambda<1-\delta.$ 

We have all the elements to state our first main result.
\begin{teo}\label{th:1}
Let $\{X(t), t\geq 0\}$ be a positive $1/\alpha$-self-similar Markov
process with increasing paths. The following assertions are
equivalent:
\begin{itemize}
\item[(i)]The subordinator $\xi,$
associated to $X$ via Lamperti's transformation, has Laplace
exponent $\phi:\re^+\to\re^+,$ which is regularly varying at $0$
with an index $\beta\in[0,1].$
\item[(ii)] Under $\p_{1}$ the random variables
$\left\{\log(X(t)/t^{1/\alpha})/\log(t), t > 1\right\}$ converge weakly as $t\to\infty$ towards a r.v. $V.$
\item[(iii)] For any $x>0,$ under $\p_{x}$ the random variables
$\left\{\log(X(t)/t^{1/\alpha})/\log(t), t > 1\right\}$ converge weakly as $t\to\infty$ towards a r.v. $V.$ 
\end{itemize}
In this case, the law of $V$ is determined in terms of the value of $\beta$ as follows: $V=0$ a.s. if
$\beta=1;$ $V=\infty,$ a.s. if $\beta=0,$ and if $\beta\in ]0,1[,$
its law has a density given by
$$\frac{\alpha^{1-\beta}2^{\beta}\sin(\beta\pi)}{\pi}v^{-\beta}(2+\alpha v)^{-2}\mathrm{d}v,\qquad v>0.$$
\end{teo}
We will see in the proof
of Theorem~\ref{th:1} that under the assumption of regular variation
of $\phi$ at $0,$  the asymptotic behaviour of $X(t)$ is quite
irregular. Namely, it is not of order $t^{a}$ for any $a>0,$ see
Remark~\ref{remark:degeneracy}. This justifies our choice of
smoothing the paths of $X$ by means of the logarithm. 

The proof of this Theorem uses among other tools the Dynkin-Lamperti Theorem for
subordinators, see e.g.~\cite{bertoinsub}. Actually, we can find some similarities between the Dynkin-Lamperti Theorem and our Theorem 1. For example, the conclusions of the former hold if and only if one of the conditions of the latter hold; both theorems describe the asymptotic behaviour of $\xi$ at a sequence of stopping times, those appearing in the former are the  first passage times above a barrier, while in the latter they are given by $\tau(\cdot).$ It shall be justified in Section \ref{FC} that in fact both families of stopping times bear similar asymptotic behaviours. 

Besides, the equivalence between (ii) and (iii) in Theorem~\ref{th:1} is a simple consequence of the scaling property. Another simple consequence of the scaling property is that: \textit{if there exists a normalizing function $h:\re^{+}\to\re^{+}$ such that for any $x>0,$ under $\p_{x},$ the random variables
$\left\{\log(X(t)/t^{1/\alpha})/h(t), t > 0\right\}$ converge weakly as $t\to\infty$ towards a non-degenerated random variable $V$ whose law does not depend on $x,$ then the function $h$ is slowly varying at infinity.}  In view of this, it is natural to ask, in the case where the Laplace exponent is not regularly varying at $0,$ if there exists a function $h$ that growths faster or slower than $\log(t)$ and such that $\log(X(t)/t^{1/\alpha})/h(t)$ converges in law to a non-degenerated random variable? The following result answers this question negatively.   

\begin{teo}\label{teo:1bis}
Assume that the Laplace exponent of $\xi$ is not regularly varying at $0$ with a strictly positive index and let $h:\re^+\to\re^+$ be a function. %that varies slowly at $\infty.$  
If $h(t)/\log(t)$ tends to $0$ or $\infty,$ as $t\to \infty,$ and  the law of $\log(X(t)/t^{1/\alpha})/h(t),$ under $\p_{1},$ converges weakly to a real valued r.v., as $t\to\infty,$ then the limiting random variable is degenerated. 
\end{teo} 

Now, observe that, in the case where the underlying subordinator has finite mean, Proposition \ref{prop:1} provides some information about the rate of growth  of the random clock $(\tau(t), t\geq 0)$ because it is equal to the additive functional $\int^{t}_{0}X^{-\alpha}_{s}\mathrm{d}s,$ $t\geq 0.$ Besides, in the case where $\phi$ is regularly varying at $0$ with an index in $[0,1[$ it can be verified that $$\frac{1}{\log(t)}\int^t_{0}X^{-\alpha}_{s}\mathrm{d}s\xrightarrow[t\to\infty]{} 0,\qquad \p_{1}-\text{a.s.}$$ see Remark \ref{remergodicthm} below. Nevertheless, in the latter case we can establish an estimate of the Darling-Kac type for the functional $\int^t_{0}X^{-\alpha}_{s}\mathrm{d}s,$  $t\geq 0,$ which provides some insight about the rate of growth of the random clock. This is the content of the following result.
\begin{prop}\label{DKprop}
The following conditions are equivalent:
\begin{itemize}
\item[(i)] $\phi$ is regularly varying at $0$ with an index $\beta\in[0,1].$
\item[(ii)] The law of $\phi\left(\frac{1}{\log(t)}\right)\int^t_{0}X^{-\alpha}_{s}\mathrm{d}s,$ under $\p_{1},$ converges in distribution, as $t\to\infty,$ to a r.v. $\alpha^{-\beta}W,$ where $W$ is a r.v. that follows a Mittag-Leffler law of parameter $\beta\in[0,1].$ 
\item[(iii)] For some $\beta\in[0,1],$ $\e_{1}\left(\left(\phi\left(\frac{1}{\log(t)}\right)\int^t_{0}X^{-\alpha}_{s}\mathrm{d}s\right)^n\right)$ converges towards $\alpha^{-\beta n}n!/\Gamma(1+n\beta),$ for $n=0,1,\ldots,$ as $t\to\infty.$  
\end{itemize} 
\end{prop}
Before continuing with our exposition about the asymptotic results for $\log(X)$ let us make a digression to remark that this result has an interesting consequence for a class of r.v. introduced by Bertoin and Yor \cite{BYfactorizations} that we next explain. Recently, Bertoin and Yor proved that there exists a $\re^+$ valued r.v. $R_{\phi}$ associated to $I_{\phi}:=\int^{\infty}_{0}\exp\{-\alpha\xi_{s}\}\mathrm{d}s,$ such that  $$R_{\phi}I_{\phi}\stackrel{\text{Law}}{=}E, \quad\text{where $E$ follows an exponential law of parameter $1$.}$$  The law of $R_{\phi}$ is completely determined by its entire moments, which in turn are given by $$\er(R_{\phi}^n)=\prod^n_{k=1} \phi(\alpha k),\qquad \text{for} \ n=1,2,\ldots$$ Furthermore, the law of $R_{\phi}$ is related to $X$ by the following formula $$\e_{1}\left(X^{-\alpha}_{s}\right)=\er(e^{-sR_{\phi}}),\qquad s\geq 0.$$ It follows therefrom that $$\e_{1}\left(\int^t_{0}X^{-\alpha}_{s}\mathrm{d}s\right)=\int_{[0,\infty[}\frac{1-e^{-tx}}{x}\pr(R_{\phi}\in\mathrm{d}x),\qquad t\geq 0.$$ These relations allow us to establish the following corollary.
\begin{cor}\label{corDK}
Assume that $\phi$ is regularly varying at $0$ with index $\beta\in[0,1].$  The following estimates 
$$\er\left(1_{\{R_{\phi}>s\}}\frac{1}{R_{\phi}}\right)\sim\frac{1}{\alpha^\beta\Gamma(1+\beta)\phi(1/\log(1/s))},\quad \pr(R_{\phi}<s)=o\left(\frac{s}{\alpha^{\beta}\Gamma(1+\beta)\phi(1/\log(1/s))}\right),$$ as $s\to 0,$ hold. If furthermore, the function $\lambda/\phi(\lambda),$ $\lambda>0,$ is the Laplace exponent of a subordinator then $$\er\left(1_{\{I_{\phi}>s\}}\frac{1}{I_{\phi}}\right)\sim\frac{\alpha^{\beta}\log(1/s)\phi(1/\log(1/s))}{\Gamma(2-\beta)},\quad \pr(I_{\phi}<s)=o\left(\frac{\alpha^{\beta}s\log(1/s)\phi\left(1/\log(1/s)\right)}{\Gamma(2-\beta)}\right),$$ as $s\to 0.$
\end{cor}
It is known, \cite{SV} Theorem 2.1, that the Laplace exponent $\phi$ is such that the function $\lambda/\phi(\lambda)$ is the Laplace exponent of a subordinator if and only if the renewal measure of $\xi$ has a decreasing density;  see also \cite{hawkes} Theorem 2.1 for a sufficient condition on the L\'evy measure for this to hold. The relevance of the latter estimates relies on the fact that in the literature about the subject there are only a few number of subordinators for which estimates for the left tail of $I_{\phi}$ are known. 

In the following theorem, under the assumption that (i) in Theorem~\ref{th:1} holds, we obtain a law of iterated logarithm for $\{\log(X(t)), t\geq 0\}$ and provide an integral test to determine the upper functions for it.

\begin{teo}\label{th:2}
Assume that the condition (i) in Theorem~\ref{th:1} above holds with
$\beta\in]0,1[.$ We have the following estimates of $\log(X(t)).$
\begin{itemize}
\item[(a)]
$\displaystyle\liminf_{t\to\infty}\frac{\log\left(X(t)\right)}{\log(t)}=1/\alpha,\qquad
\p_1-\text{a.s.}$
\item[(b)] Let $g:]e,\infty[\to\re^+$ be the function defined by
$$g(t)=\frac{\log\left(\log(t)\right)}{\varphi\left(t^{-1}\log\left(\log(t)\right)\right)},\qquad t>e,$$
with $\varphi$ the right continuous inverse of $\phi.$ For any increasing function
$f$ with positive increase, i.e.
$0<\liminf_{t\to\infty}\frac{f(t)}{f(2 t)},$ we have that
\begin{equation}\label{limsup}\limsup_{t\to\infty}\frac{\log(X(t))}{f\left(\log(t)\right)}=0,\qquad
\text{or} \qquad =\infty,\qquad \p_1-\text{a.s.}\end{equation}
according whether
\begin{equation}\label{inttest}\int^{\infty}\phi\left(1/f(g(t))\right)\mathrm{d}t<\infty, \qquad
\text{or}\qquad =\infty.\end{equation}
\end{itemize}
\end{teo}
\begin{remark}\label{remark0}
Observe that in the case where the Laplace exponent varies regularly
at $0$ with index  $1,$ then Theorem~\ref{th:1} implies
that
$$\frac{\log\left(X(t)\right)}{\log(t)}\xrightarrow[t\to\infty]{\text{Probability}}
1/\alpha.
$$ A question that remains open is what are NASC for this convergence to hold almost
surely? Proposition~\ref{prop:1} says that the finiteness of the mean of the underlying subordinator is a sufficient condition for this to hold. So it should be determined if this condition is also  necessary?
\end{remark}
\begin{remark}
In the case where $\phi$ is slowly varying at $0,$ Theorem \ref{th:1} implies that $$\frac{\log\left(X(t)\right)}{\log(t)}\xrightarrow[t\to\infty]{\text{Probability}}
\infty.
$$ In the proof of Theorem \ref{teo:1bis} it will be seen that if $h:\re^{+}\to]0,\infty[$ is a function such that $\log(t)/h(t)\to 0$ as $t\to \infty,$ then $$\frac{\log\left(X(t)\right)}{h(t)}\xrightarrow[t\to\infty]{\text{Probability}}
0.
$$ Which is a weak analogue of Theorem \ref{th:2}.
\end{remark}

\begin{remark}\label{remark1}
Observe that the local behaviour of $X,$ when started at a strictly positive point, is quite similar to that of the underlying subordinator. This is due to the elementary fact $$\frac{\tau(t)}{t}\xrightarrow[t\to 0+]{} 1, \qquad  \p_{1}-\text{a.s.}$$ So, for short times the behaviour of $\xi$ is not affected by the time change, which is of course not the case for large times. Using this fact and known results for subordinators, precisely Theorem~3 in \cite{bertoinbook}~Section~III.3, it is straightforward to prove the following Proposition which is the analogous in short time of our Theorem~\ref{th:1}. We omit the details of the proof. 
\end{remark}
\begin{prop}\label{cor:frag1}
Let $\{X(t), t\geq 0\}$ be a positive $1/\alpha$-self-similar Markov
process with increasing paths. The following conditions are
equivalent:
\begin{itemize}
\item[(i)]The subordinator $\xi,$
associated to $X$ via Lamperti's transformation, has Laplace
exponent $\phi:\re^+\to\re^+,$ which is regularly varying at $\infty$
with an index $\beta\in]0,1[.$
\item[(ii)] There exists an increasing function $h:\re^+\to\re^+$ such that under $\p_{1}$ the random variables
$\left\{h(t)\log(X(t)), t > 0\right\}$ converge weakly as $t\to 0$ towards a non-degenerated r.v. 
\item[(ii)] There exists an increasing function $h:\re^+\to\re^+$ such that under $\p_{1}$ the random variables
$h(t)\left(X(t)-1\right), t>0$ converge weakly as $t\to 0$ towards a non-degenerated r.v.
\end{itemize}
In this case, the limit law is a stable law with parameter $\beta,$ and $h(t)\sim\varphi(1/t),$ as $t\to 0,$ with $\varphi$ the right continuous inverse of $\phi.$ 
\end{prop}
It is also possible to obtain an  analogue of Theorem~\ref{th:2}, which is a simple translation for pssMp of results such as those appearing in \cite{bertoinbook}~Section~III.4. A study of the short and large time behaviour of $X$ under $\p_{0+}$ has been done in~\cite{rivero2003} and \cite{CP}.  

To finish this section we mention that our Theorem \ref{th:1} has an interesting application to self-similar fragmentation processes that we next describe. We first provide a few definitions from fragmentation theory, and refer to the recent book ~\cite{bertoinFC} for background on it. First, we introduce the set
$$\mathcal{S}^{\downarrow}=\left\{\mathbf{s}=(s_{i})_{i\in\nr}: 1>s_{1}\geq s_{2}\geq\cdots\geq 0,\ \sum^{\infty}_{i=1}s_{i}=1 \right\}.$$ Let $Y=(Y(t),t\geq 0)$  denote a random process such that $$Y(t)=\left(Y_{1}(t),Y_{2}(t),\ldots\right)\in\mathcal{S}^{\downarrow},\quad \forall t\geq 0,\ \text{a.s.}$$ We suppose that $Y$ is continuous in probability and Markovian. For every $r\geq 0,$ we denote by $\mathbb{Q}_{r}$ the law of $Y$ started from the configuration $(r,0,\ldots),$ and assume that $\mathbb{Q}_{0}(Y_{1}(t)=0)=1.$ It is said that $Y$ is a self-similar fragmentation process if:
\begin{itemize}
\item There exists $\alpha\in\re,$ called index of self-similarity, such that for every $r>0$ the distribution under $\mathbb{Q}_{1}$ of the rescaled process $\left(rY(r^{\alpha}t),t\geq 0\right)$ is $\mathbb{Q}_{r}.$
\item For every $\mathbf{s}=(s_{1},s_{2},\ldots)\in\mathcal{S}^{\downarrow},$ if $(Y^{(i)}, i\in\nr)$ is a sequence of independent processes such that $Y^{(i)}$ has the law $\mathbb{Q}_{s_{i}},$ and if $\widetilde{Y}$ denotes the decreasing rearrangement of the family $(Y^{(i)}_{n}(t): i,n\in\nr),$ then $\widetilde{Y}=(\widetilde{Y}(t), t\geq 0)$ is a version of $Y$ started from configuration $\mathbf{s}.$ 
\end{itemize}
It is known that associated to $Y$ there exists a characteristic triple $(\kappa,c,\nu),$ where $\kappa\geq 0,$ is the killing rate, $c\geq 0$ is known as the erosion coefficient and $\nu$ is the so called splitting measure, which is a measure over $\mathcal{S^{\downarrow}}$ such that $$\int_{\mathcal{S}^{\downarrow}}\nu(\mathrm{d}s)(1-s_{1})<\infty.$$ Here we will only consider self-similar fragmentations with self-similarity index $\alpha>0,$ and killing and erosion coefficient $\kappa=0=c.$ In \cite{beasymp}, Bertoin studied under some assumptions the long time behaviour of the process $Y$ via an empirical probability measure carried, at each $t,$ by the components of $Y(t)$
\begin{equation}\label{empmeasure}\widetilde{\rho}_{t}(\mathrm{d}y)=\sum_{i\in \nr}Y_{i}(t)\delta_{t^{1/\alpha}Y_{i}(t)}(\mathrm{d}y),\qquad t\geq 0.\end{equation}  To be more precise, he proved that if the function $$\Phi(q):=\int_{\mathcal{S}^{\downarrow}}\left(1-\sum^{\infty}_{i=1}s^{q+1}_{i}\right)\nu(\mathrm{d}s),\qquad q>0,$$ is such that $m:=\Phi'(0+)<\infty,$ then the measure defined in (\ref{empmeasure}) converges in probability to a deterministic measure, say $\widetilde{\rho}_{\infty},$ which is completely determined by its entire moments 
$$\int^\infty_{0}x^{\alpha k}\widetilde{\rho}_{\infty}(\mathrm{d}x)=\frac{(k-1)!}{\alpha\mu \Phi(\alpha)\cdots \Phi(\alpha(k-1))},\qquad k=1,2,\ldots $$ with the assumption that the latter quantity equals $(\alpha m)^{-1},$ when $k=1.$ Bertoin proved this result by cleverly applying the results in \cite{BC2002} and the fact that there exists a $1/\alpha$-increasing positive self-similar Markov process, say $\widetilde{Z}= \left(\widetilde{Z}_{t}, t\geq 0\right),$ called the \textit{process of the tagged fragment} such that $\mathbb{Q}_{r}(\widetilde{Z}_{0}=r)=1,$ and for any bounded and measurable function $f:\re^+\to\re^+$
$$\mathbb{Q}_{1}\left(\widetilde{\rho}_{t}f\right)=\mathbb{Q}_{1}\left(\sum^\infty_{i=1}Y_{i}(t)f(t^{1/\alpha}Y_{i}(t))\right)=\mathbb{Q}_{1}\left(f\left(t^{1/\alpha}/\widetilde{Z}_{t}\right)\right),\qquad t\geq 0;$$ and that the process $\widetilde{Z}$ is an increasing $1/\alpha$--pssMp whose underlying subordinator has Laplace exponent $\Phi.$ Besides, it can be viewed, using the method of proof of Bertoin, that if $\Phi'(0+)=\infty,$ then the measure $\widetilde{\rho}_{t}$ converges in probability to the law of a r.v. degenerated at $0.$ This suggest that in the latter case, to obtain further information about the repartition of the components of $Y(t)$ it would be convenient to study a different form of the empirical measure of $Y.$ A suitable form of the empirical measure is given by the random probability measure $$\rho_{t}(\mathrm{d}y)=\sum^\infty_{i=1}Y_{i}(t)\delta_{\{\log(Y_{i}(t))/\log(t)\}}(\mathrm{d}y),\qquad t\geq 0.$$  The arguments provided by Bertoin are quite general and can be easily modified to prove the following consequence of Theorem \ref{th:1}, we omit the details of the proof.

\begin{cor}\label{cor:frag}
Let $Y$ be a self-similar fragmentation with self-similarity index $\alpha>0$ and characteristics $(0,0,\nu).$ Assume that the function $\Phi$ is regularly varying at $0$ with an index $\beta\in[0,1].$ Then, as $t\to\infty,$ the random probability measure $\rho_{t}(\mathrm{d}y)$ converges in probability towards the law of $-\alpha^{-1}-V,$ where $V$ is as in Theorem \ref{th:1}.  
\end{cor}

This result extend to the case with infinite mean some results in Brennan and Durrett in~\cite{Brennandurrett2,Brennandurrett1}. This will be made precise in Section \ref{sect:frag}, where we will provide necessary conditions on the splitting measure for the function $\Phi$ to satisfy the hypotheses in Corollary \ref{cor:frag}.

The rest of this note is devoted to prove the results stated before. The document is organized so that each subsequent Section contains a proof. The last section of this work is constituted of a comparison of the results here obtained and those known that describe the behaviour of the underlying subordinator.  

\end{section}

\begin{section}{Proof of Proposition~\ref{prop:1}}\label{proofprop:1}
Assume that the mean of $\xi$ is finite, $m:=\er(\xi_1)<\infty$.
According to the result of Bertoin and Caballero there exists a
measure $\p_{0+}$ on the space of c\`adl\`ag paths defined over
$]0,\infty[$ that takes only positive values, under which the
canonical process is a strong Markov process with the same semigroup
as $X$ and its entrance law can be described by means of the
exponential functional
$I=\int^{\infty}_0\exp\{-\alpha\xi_s\}\mathrm{d}s,$ by the formula
\begin{equation*}
\e_{0+}\left(f(X(t))\right)=\frac{1}{\alpha
m}\er\left(f\left(\left(t/I\right)^{1/\alpha}\right)\frac{1}{I}\right),\qquad
t>0,
\end{equation*} for any measurable function $f:\re^+\to\re^+.$
A straightforward consequence of the scaling property is that the
process of the Ornstein-Uhlenbeck type $U$ defined by
$$U_t=e^{-t/\alpha}X(e^t),\qquad t\in\re,$$ under $\e_{0+}$ is a strictly stationary
process. This process has been studied by Carmona, Petit and
Yor~\cite{cpy97} and by Rivero in~\cite{rivero2003}. Therein, is
proved that $U$ is a positive recurrent and strong Markov process.
Observe that the law of $U_0$ under $\e_{0+},$ is given by the
probability measure $\mu$ defined in~(\ref{eq:limitlaw}). By the ergodic theorem we have that
$$\frac{1}{t}\int^{t}_0f(U_s)\mathrm{d}s\xrightarrow[t\to\infty]{}\e_{0+}\left(f(U_0)\right)=\mu(f),\quad \p_{0+}-\text{a.s.}$$
for every function $f\in L^{1}(\mu)$. Observe that a change
of variables $u=e^{s}$ in the latter equation allows to deduce that
$$\frac{1}{\log(t)}\int^{t}_1f(u^{-1/\alpha}X(u))\frac{\mathrm{d}u}{u}=\frac{1}{\log(t)}\int^{\log(t)}_0f(U_s)\mathrm{d}s\xrightarrow[t\to\infty]{}\e_{0+}\left(f(U_0)\right),\quad \p_{0+}-\text{a.s.}$$

Now to prove the second assertion of Proposition~\ref{prop:1} we use
the well known fact that
$$\lim_{t\to\infty}\frac{\xi_t}{t}=m,\qquad
\pr-\text{a.s.}$$ So, to prove the result it will be sufficient to
establish that
\begin{equation}\label{equation:convtau}\tau(t)/\log(t)\xrightarrow[t\to\infty]{}1/m\alpha,\qquad\pr-\text{a.s.}\end{equation}
Because in that case
$$\frac{\log(X(t))}{\log(t)}=\frac{\xi_{\tau(t)}}{\tau(t)}\frac{\tau(t)}{\log(t)}\xrightarrow[t\to\infty]{}m/\alpha m,\qquad\pr-\text{a.s.}$$
Now, a simple consequence of Lamperti's transformation is that
$$\tau(t)=\int^{t}_0X^{-\alpha}_s\mathrm{ds}=\int^{t}_0\left(s^{-1/\alpha}X(s)\right)^{-\alpha}\frac{\mathrm{ds}}{s},\qquad t\geq 0.$$
So, the result just proved applied to the function
$f(x)=x^{-\alpha},$ $x>0,$ leads
$$\frac{1}{\log(1+t)}\int^{1+t}_1\left(u^{-1/\alpha}X(u)\right)^{-\alpha}\frac{\mathrm{d}u}{u}\xrightarrow[t\to\infty]{}1/\alpha m,\quad \p_{0+}-\text{a.s.}$$
Denote by $\mathcal{H}$ the set were the latter convergence holds.
By the Markov property it is clear that
$$\p_{0+}\left(\p_{X(1)}\left(\frac{1}{\log(1+t)}\int^{t}_0\left(u^{-1/\alpha}X(u)\right)^{-\alpha}\frac{\mathrm{d}u}{u}\nrightarrow 1/\alpha m\right)\right)=\p_{0+}\left(\mathcal{H}^{c}\right)=0.$$
So for $\p_{0+}$--almost every $x>0,$
$$\p_{x}\left(\frac{1}{\log(1+t)}\int^{t}_0\left(u^{-1/\alpha}X(u)\right)^{-\alpha}\frac{\mathrm{d}u}{u}\xrightarrow[t\to\infty]{}1/\alpha m\right)=1.$$
For such an $x,$ it is a consequence of the scaling property that
$$\frac{1}{\log(1+t)}\int^{t}_0\left(u^{-1/\alpha}xX(ux^{-\alpha})\right)^{-\alpha}\frac{\mathrm{d}u}{u}\xrightarrow[t\to\infty]{}1/\alpha m,\qquad \p_{1}-\text{a.s.}$$ Therefore, by making a change of variables $s=ux^{-\alpha}$ and using the fact that $\frac{\log(1+tx^{-\alpha})}{\log(t)}\to 1,$
as $t\to\infty$, we prove that (\ref{equation:convtau}) holds. Which in view of the previous comments finish the
proof of the second assertion in Proposition~\ref{prop:1}.

\begin{remark}\label{remergodicthm} In the case where the mean is
infinite, $\er(\xi_1)=\infty,$ we can still construct a measure $N$
with all but one of the properties of $\p_{0+};$ the missing property
is that $N$ is not a probability measure, it is in fact a
$\sigma$-finite, infinite measure. The measure $N$ is constructed
following the methods used by Fitzsimmons~\cite{Fitzsimmonsrecext}, the details of this construction are beyond the scope of this note so we omit them. Thus, using results from the infinite
ergodic theory (see e.g. \cite{Aaronson} Section 2.2) it can be
verified that
$$\frac{1}{\log(t)}\int^{t}_0f(s^{-1/\alpha}X(s))\frac{\mathrm{d}s}{s}\xrightarrow[t\to\infty]{}0,\quad
\p_{1}-\text{a.s.}$$ for every function $f$ such that
$N(|f(X(1))|)=\mu|f|<\infty;$ in particular for $f(x)=x^{-\alpha},$ $x>0.$
\end{remark}
\end{section}

\begin{section}{Proof of Theorem~\ref{th:1}}\label{proofth:1}
The proof of Theorem~\ref{th:1} follows the method of proof in~\cite{BC2002}. So, here we will first explain how the auxiliary Lemmas and Corollaries in~\cite{BC2002} can be extended in our setting and then we will apply those facts to prove the claimed results. 

We start by introducing some notation. We define the processes of the age and rest of life associated to the subordinator $\xi,$
$$(A_t, R_t)=(t-\xi_{L(t-)},\xi_{L(t)}-t),\qquad t\geq 0,$$ where
$L(t)=\inf\{s>0 : \xi_s>t\}.$  The methods used by Bertoin and
Caballero are based on the fact that if the mean $\er(\xi_1)<\infty$
then the random variables $(A_t, R_t)$ converge weakly to a
non-degenerated random variable $(A,R)$ as the time tends to
infinity. In our setting, $\er(\xi_1)=\infty,$ the random variables
$(A_t, R_t)$ converge weakly towards $(0,\infty).$ Nevertheless, if
the Laplace exponent $\phi$ is regularly varying at $0$ then
$(A_t/t, R_t/t)$ converge weakly towards a non-degenerated random
variable $(U,O)$ (see e.g. Theorem~3.2 in~\cite{bertoinsub} or
Lemma~\ref{lemma:bc2} below, for a precise statement). This fact,
known as the Dynkin-Lamperti theorem, is the clue to solve our
problem.

The following results can be proved with little effort following Bertoin and Caballero. For $b>0,$ let $T_{b}$ be the first entry time into $]b,\infty[$ for $X,$ viz. $T_{b}=\inf\{s>0 : X(s)>b\}.$

\begin{lemma}\label{lemma:bc1}
Fix $0<x<b.$ The distribution of the pair $(T_{b}, X(T_b))$ under
$\p_x$ is the same as that of
$$\left(b^{\alpha}\exp\{-\alpha A_{\log\left( b/x \right)}\}\int^{L(\log\left( b/x \right))}_0\exp\{-\alpha\xi_s\}\mathrm{d}s, b \exp\{R_{\log(b/x)}\} \right).$$
\end{lemma}
This result was obtained in~\cite{BC2002} as Corollary 5 and still is true under our assumptions because its proof holds without any hypothesis on the mean of the underlying subordinator. Now, using the latter result, the arguments in the proof of Lemma 6 in~\cite{BC2002}  and the Dynkin-Lamperti Theorem for subordinators we obtain the following result. 
\begin{lemma}\label{lemma:bc2}
Assume that the Laplace exponent $\phi$ of the subordinator $\xi$ is
regularly varying at infinity with index $\beta\in[0,1].$ Let
$F:\mathbf{D}_{[0,s]}\to \re$ and $G:\re^{2}_+\to\re$ be measurable
and bounded functions. Then
\begin{equation}
\lim_{t\to\infty}\er\left(F\left(\xi_r, r\leq
s\right)G\left(\frac{A_t}{t},
\frac{R_t}{t}\right)\right)=\er\left(F(\xi_r, r\leq
s)\right)\er\left(G\left(U, O\right)\right),
\end{equation}
Where $(U,O)$ is a $[0,1]\times[0,\infty]$ valued random variable
whose law is determined as follows: if $\beta=0$ (resp. $\beta=1$),
it is the Dirac mass at $(1,\infty)$ (resp. at $(0,0))$. For
$\beta\in]0,1[,$ it is the distribution with density
$$p_{\beta}(u,w)=\frac{\beta\sin\beta\pi}{\pi}(1-u)^{\beta-1}(u+w)^{-1-\beta},\qquad
0<u<1, w>0.$$
\end{lemma}
Finally, using arguments similar to those provided in the proof of Corollary 7 in ~\cite{BC2002} we deduce from the latter and former results the following Lemma.
\begin{lemma}\label{lemma:bc3}
Assume that the Laplace exponent $\phi$ of the subordinator $\xi$ is
regularly varying at infinity with index $\beta\in[0,1].$ Then as
$t$ tends to $\infty$ the triplet
$$\left( \int^{L(t)}_{0}\exp\{-\alpha\xi_s\}\mathrm{d}s, \frac{A_t}{t},
\frac{R_t}{t}\right)$$ converges in distribution towards
$$\left( \int^{\infty}_{0}\exp\{-\alpha\xi_s\}\mathrm{d}s, U, O\right),$$
Where $\xi$ is independent of the pair $(U, O)$ which has the law
specified in the Lemma~\ref{lemma:bc2}.
\end{lemma}
We have the necessary tools to prove Theorem~\ref{th:1}.
\begin{proof}[Proof of Theorem~\ref{th:1}]
Let $c>-1,$ and $b(x)=e^{c\log(1/x)},$ for $0<x<1.$ In the case where $\beta=1$ we will furthermore assume that $c\neq 0$ owing that in this setting $0$ is a point of discontinuity for the distribution of $U.$ The elementary relations
$$\log\left(b(x)/x\right)=(c+1)\log(1/x),\qquad \log\left(b(x)/x^2\right)=(c+2)\log(1/x),\qquad 0<x<1,$$
will be useful.
The following equality in law follows from Lemma~\ref{lemma:bc1} 
\begin{equation}\label{eq:equalityinlaw}
\begin{split}
&\left(\frac{\log\left(T_{b(x)/x}\right)}{\log\left(1/x\right)}, \frac{\log\left(X\left(T_{b(x)/x}\right)\right)}{\log(1/x)}\right)\stackrel{\text{Law}}{=}\\
&\left( \frac{\alpha\log(b(x)/x)-\alpha A_{\log\left(b(x)/x^2\right)}+\log\left(\int^{L\left(\log\left( b(x)/x^2 \right)\right)}_0\exp\{-\alpha\xi_s\}\mathrm{d}s\right)}{\log\left(1/x\right)}, \frac{\log\left(b(x)/x\right)+R_{\log\left(b(x)/x^2\right)}}{\log\left(1/x\right)}\right),
\end{split}
\end{equation}for all $0<x<1.$ Moreover, observe that the r.v. $\int^{L(r)}_0\exp\{-\alpha\xi_s\}\mathrm{d}s$ converges almost surely to $\int^{\infty}_0\exp\{-\xi_s\}\mathrm{d}s,$ as $r\to\infty;$ and that for any $t>0$ fixed,
$$\p_1\left(\frac{\log\left(x X(tx^{-\alpha})\right)}{\log(1/x)}>c\right)=\p_1\left(x X(tx^{-\alpha})>b(x)\right),\qquad 0<x<1,$$
\begin{equation}\label{eq:importantineq1}
\begin{split}
\p_1(T_{b(x)/x}<tx^{-\alpha})&\leq \p_1(xX(tx^{-\alpha})>b(x))\\
&\leq\p_1(T_{b(x)/x}\leq tx^{-\alpha})\leq \p_1(xX(tx^{-\alpha})\geq b(x)),\quad 0<x<1.
\end{split}
\end{equation}
Thus, under the assumption of regular variation at $0$ of $\phi,$  the equality in law in (\ref{eq:equalityinlaw}) combined with the result in Lemma~\ref{lemma:bc3} lead to the weak convergence
\begin{equation}\label{eq:weakconv}
\left(\frac{\log(T_{b(x)/x})}{\log(1/x)},
\frac{\log\left(X\left(T_{b(x)/x}\right)\right)}{\log(1/x)}\right)\xrightarrow[x\to
0+]{\textrm{D}}\left(\alpha\left[c+1-(c+2)U\right],
c+1+(c+2)O\right).
\end{equation}
%Observe that the limiting distribution for $\log\left(X\left(T_{b(x)/x}\right)\right)/\log(1/x),$ is continuous in $\re$ (res\-pecti\-vely in $\re\setminus\{0\}$ if $\beta=1$). 
As a consequence we get
$$\p_1(T_{b(x)/x}<tx^{-\alpha})=\p_1\left(\frac{\log\left(T_{b(x)/x}\right)}{\log\left(1/x\right)}<\frac{\log(t)}{\log\left(1/x\right)}+\alpha\right)\xrightarrow[x\to 0+]{}\pr\left(\frac{c}{c+2}< U\right),$$ for  $c>-1.$
Which in view of the first two inequalities in (\ref{eq:importantineq1}) shows that for any $t>0$ fixed
$$\p_1\left(\frac{\log\left(x
X(tx^{-\alpha})\right)}{\log(1/x)}>c\right)\xrightarrow[x\to\
0+]{\textrm{}}\pr\left(\frac{c}{c+2} < U\right),$$
for $c>-1,$ and we have so proved that (i) implies (ii).

\noindent Next, we prove that (ii) implies (i). If (ii) holds then
\begin{equation*}
\p_1\left(\frac{\log\left(x
X(tx^{-\alpha})\right)}{\log(1/x)}>c\right)\xrightarrow[x\to\
0+]{\textrm{}}\pr\left( V > c\right),
\end{equation*}for every $c>-1$ point of continuity of the distribution of $V.$
Using this and the second and third inequalities in (\ref{eq:importantineq1}) we obtain that
\begin{equation*}
\p_1\left(\frac{\log\left(T_{b(x)/x}\right)}{\log\left(1/x\right)}<\frac{\log(t)}{\log\left(1/x\right)}+\alpha\right)\xrightarrow[x\to 0+]{}\pr\left(c< V\right).
\end{equation*}
Owing to the equality in law~(\ref{eq:equalityinlaw}) we have that
\begin{equation}\label{eq:wlimitundershoot}
\begin{split}
&\pr\left(c< V\right)\\
&=\lim_{x\to 0+}\pr\left( \frac{\alpha\log(b(x)/x)-\alpha A_{\log\left(b(x)/x^2\right)}+\log\left(\int^{L\left(\log\left( b(x)/x^2 \right)\right)}_0\exp\{-\alpha\xi_s\}\mathrm{d}s\right)}{\log\left(1/x\right)}<\frac{\log(t)}{\log\left(1/x\right)}+\alpha\right)\\
&=\lim_{x\to 0+}\pr\left( \alpha(c+1)-\frac{\alpha(c+2) A_{\log\left(b(x)/x^2\right)}}{\log\left(b(x)/x^2\right)}<\alpha\right)\\
&=\lim_{z\to \infty}\pr\left( \frac{A_{z}}{z}>\frac{c}{c+2}\right)
\end{split}
\end{equation}
So we can ensure that if (ii) holds then $A_{z}/z$ converges weakly, as $z\to \infty,$ which is well known to be equivalent to the regular variation at $0$ of the Laplace exponent $\phi,$ see e.g.~\cite{bertoinbook} Theorem~III.2. We have so proved that (ii) implies (i).

\noindent To finish, observe that if (i) holds with $\beta=0,$ it is
clear that $V=\infty$ a.s. given that in this case $U=1$ a.s. In
the case where (i) holds with $\beta\in]0,1]$ the limit r.v. $V$ has
the same law as the random variable $2U/\alpha(1-U),$ and an
elementary calculation proves that $V$ has the law described in
Theorem \ref{th:1}.
\end{proof}

\begin{remark}\label{remark:degeneracy}
Observe that if in the previous proof we replace the function $b$ by
$b'(x,a)=ae^{c\log(1/x)},$ for $a>0,$ $c>-1$ and  $0<x<1,$ then
$$\p_{1}(x^{1+c}X(x^{-\alpha})>a)=\p_x(X(1)>b'(x,a))=\p_1\left(\frac{\log\left(x
X(x^{-\alpha})\right)}{\log(1/x)}>c+\frac{a}{\log(1/x)}\right),$$
and therefore the limit of the latter quantity
does not depend on $a,$ as $x$ goes to $0+.$ That is for each $c>-1$ we have the weak
convergence of r.v.
$$x^{1+c}X(x^{-\alpha})\xrightarrow[x\to 0]{\textrm{D}}Y(c),$$ and $Y(c)$ is an $\{0,\infty\}$-valued random variable whose law is given  by
$$\p(Y(c)=\infty)=\p\left(\frac{c}{c+2}<U\right), \qquad \p(Y(c)=0)=\p\left(\frac{c}{c+2}\geq U\right).$$
Therefore, we can ensure that the asymptotic behaviour of $X(t)$ is not of the order $t^{a}$ for any $a>0,$ as $t\to\infty.$
\end{remark}
\end{section}
\begin{section}{Proof of Theorem~\ref{teo:1bis}}\label{proofth:1bis}
Assume that the Laplace exponent of $\xi$ is not regularly varying at $0$ with a strictly positive index. Let $h:\re^{+}\to]0,\infty[$ be a slowly varying function such that $h(t)\to\infty$ as $t\to \infty,$ and define$f(x)=h(x^{-\alpha}),$ $0<x<1.$ Assume that $h$, and so $f,$ are such that 
\begin{equation}\label{convergenciacontradiccion}\frac{\log(xX(x^{-\alpha}))}{f(x)}\xrightarrow[x\to0+]{\text{Law}}V,\end{equation} where $V$ is an a.s. non-degenerated, finite and positive valued random variable. For $c$ a continuity point of $V$ let $b_{c}(x)=\exp\{cf(x)\},$ $0<x<1.$ We have that $$\p_{1}\left(\frac{\log(xX(x^{-\alpha}))}{f(x)}>c\right)\xrightarrow[x\to 0+]{}\p(V>c).$$ Arguing as in the proof of Theorem~\ref{th:1} it is proved that the latter convergence implies that \begin{equation*}
\p_1\left(\frac{\log\left(T_{b_{c}(x)/x}\right)}{\log\left(1/x\right)}\leq\alpha\right)\xrightarrow[x\to 0+]{}\pr\left(V>c\right).
\end{equation*}
Using the identity in law (\ref{eq:equalityinlaw}) and arguing as in equation (\ref{eq:wlimitundershoot}) it follows that the latter convergence implies that 
\begin{equation}\label{eq:limit2}
\begin{split}
\pr(V> c)&=\lim_{x\to 0+}\pr\left(\frac{A_{\log(b_{c}(x)/x^2)}}{f(x)}\geq c\right)\\
&=\lim_{x\to 0+}\pr\left(\frac{A_{\log(b_{c}(x)/x^2)}}{\log(b_{c}(x)/x^2)}\left(c+\frac{2\log(1/x)}{f(x)}\right)\geq c\right),
\end{split}
\end{equation}
where the last equality follows from the definition of $b_{c}.$ 

Now, assume that $\frac{\log(t)}{h(t)}\to 0,$ as $t\to\infty,$ or equivalently that $\frac{\log(1/x)}{f(x)}\to 0,$ as $x\to 0+.$ It follows that 
\begin{equation}
\begin{split}
\pr(V> c)&=\lim_{x\to 0+}\pr\left(\frac{A_{\log(b_{c}(x)/x^2)}}{\log(b_{c}(x)/x^2)}\geq 1\right)\\
&=\lim_{z\to \infty}\pr\left(\frac{A_{z}}{z}\geq 1\right).
\end{split}
\end{equation}Observe that this equality holds for any $c>0$ point of continuity of $V.$ Making $c$ first tend to infinity and then to $0+,$ respectively, and using that $V$ is a real valued r. v. it follows that $$\pr(V=\infty)=0=\lim_{z\to \infty}\pr\left(\frac{A_{z}}{z}\geq 1\right)=\pr(V>0).$$ Which implies that $V=0$ a.s. which in turn is a contradiction to the fact that $V$ is a non-degenerated random variable. 

%Observe that even if we allow $V$ to be a extended r. v. the latter would imply that $\p(V=\infty)=1-\p(V=0).$ So, if $\frac{\log(1/x)}{f(x)}\to 0$ and the limit in equation (\ref{convergenciacontradiccion}) holds, then the limit r. v. $V$ is degenerated.

In the case where $\frac{\log(t)}{h(t)}\to \infty,$ as $t\to\infty,$ or equivalently $\frac{\log(1/x)}{f(x)}\to \infty,$ as $x\to 0+,$ we will obtain a similar contradiction. Indeed, let $l_{c}:\re^+\to\re^+$ be the function $l_{c}(x)=\log(b_{c}(x)/x^2),$ for $x>0,$ this function is strictly decreasing and so its inverse $l_{c}^{-1}$ exists. Observe that by hypothesis $\log(b_{c}(x)/x^2)/f(x)=c+\frac{2\log(1/x)}{f(x)}\to\infty$ as $x\to 0,$ thus $z/f\left(l^{-1}_{c}(z)\right)\to\infty$ as $z\to\infty.$ So, for any $\epsilon>0,$ it holds that $f\left(l^{-1}_{c}(z)\right)/z<\epsilon,$ for every $z$ large enough. It follows from the first equality in equation (\ref{eq:limit2}) that 
\begin{equation}
\begin{split}
\pr(V\geq c)&=\lim_{z\to\infty}\pr\left(\frac{A_{z}}{z}\frac{z}{f(l^{-1}_{c}(z))}\geq c\right)\\
&\geq\lim_{z\to\infty}\pr\left(\frac{A_{z}}{z}\geq c\epsilon\right),
\end{split}
\end{equation}
for any $c$ point of continuity of the distribution of $V.$ So, by replacing $c$ by $c/\epsilon,$ making $\epsilon$ tend to $0+,$ and using that $V$ is finite a.s. it follows that $$\frac{A_{z}}{z}\xrightarrow[z\to\infty]{\text{Law}}0.$$ By the Dynkin Lamperti Theorem it follows that the Laplace exponent $\phi$ of the underlying subordinator $\xi,$ is regularly varying at $0$ with index $1.$ Which is a contradiction to our assumption that  the Laplace exponent of $\xi$ is not regularly varying at $0$ with a strictly positive index.  
%By  Theorem \ref{th:1} it follows that 
%$$\frac{\log(X(t)/t^{1/\alpha})}{\log(t)}\xrightarrow[t\to\infty]{\text{Probability}} 0.$$ But by hypothesis we have that    
%$$\frac{\log(X(t)/t^{1/\alpha})}{h(t)}\xrightarrow[t\to\infty]{\text{Law}} V,$$ where $V$ is a non-degenerated r. v. Therefrom, we obtain the contradiction to the former limit in probability $$\frac{\log(X(t)/t^{1/\alpha})}{\log(t)}=\frac{h(t)}{\log(t)}\frac{\log(X(t)/t^{1/\alpha})}{h(t)}\xrightarrow[t\to\infty]{\text{Probability}} 0.$$ So, it is impossible that $V$ be non-degenerated.

\end{section}
\begin{section}{Proof of Proposition \ref{DKprop}}\label{proofDK}
We will start by proving that (i) is equivalent to
\begin{itemize}
\item[(i')]For any $r>0,$ $\displaystyle \frac{\log\left(\int^{r/\phi\left(1/t\right)}_{0}\exp\{\alpha\xi_{s}\}\mathrm{d}s\right)}{\alpha t}\xrightarrow[t\to\infty]{\text{Law}} \widetilde{\xi}_{r},$ with $\widetilde{\xi}$ a stable subordinator with self-similarity parameter $\beta$ whenever $\beta\in]0,1[,$ and in the case where $\beta=0,$ respectively $\beta=1,$ we have that $\widetilde{\xi}_{r}=\infty1_{\{e(1)<r\}},$ respectively $\widetilde{\xi}_{r}=r$ a.s. where $e(1)$ denotes an exponential r.v. with parameter $1.$ 
\end{itemize}
Indeed, using the time reversal property for L\'evy processes we obtain the equality in law
\begin{equation*}
\begin{split}
\int^{r/\phi\left(1/t\right)}_{0}\exp\{\alpha\xi_{s}\}\mathrm{d}s &=\exp\{\alpha\xi_{r/\phi(1/t)}\}\int^{r/\phi\left(1/t\right)}_{0}\exp\{-\alpha(\xi_{r/\phi(1/t)}-\xi_{s})\}\mathrm{d}s\\ 
&\stackrel{\text{Law}}{=}\exp\{\alpha\xi_{r/\phi(1/t)}\}\int^{r/\phi\left(1/t\right)}_{0}\exp\{-\alpha\xi_{s}\}\mathrm{d}s.
\end{split}
\end{equation*}
Given that the r. v. $\int^\infty_{0}\exp{-\alpha\xi_{s}}\mathrm{d}s$ is finite $\pr$-a.s., see e.g. \cite{bysurvey}, we deduce that 
$$\int^{r/\phi\left(1/t\right)}_{0}\exp\{-\alpha\xi_{s}\}\mathrm{d}s\xrightarrow[t\to\infty]{} \int^{\infty}_{0}\exp\{-\alpha\xi_{s}\}\mathrm{d}s<\infty,\qquad \pr-\text{a.s.}$$
These two facts, allow us to conclude that as $t \to \infty,$ the r.v. $\log\left(\int^{r/\phi\left(1/t\right)}_{0}\exp\{\alpha\xi_{s}\}\mathrm{d}s\right)/\alpha t$ converge in law if and only if $\xi_{r/\phi(1/t)}/t$ does. It is well known that the latter convergence holds if and only if $\phi$ is regularly varying at $0$ with an index $\beta\in[0,1].$ In this case both sequences of r.v. converge weakly towards $\widetilde{\xi}_{r}.$ 

Let $\varphi$ be the inverse of $\phi.$ Assume that (i), and so (i'), hold. To prove that (ii) holds we will use the following equalities valid for $\beta\in]0,1]$, for any $x>0$
\begin{equation}\label{eq:ML*}
\begin{split}
\pr\left(\left(\alpha\widetilde{\xi}_{1}\right)^{-\beta}<x\right)&=\pr\left(\alpha\widetilde{\xi}_{1}>x^{-1/\beta}\right)\\
&=\pr\left(\alpha\widetilde{\xi}_{x}>1\right)\\
&=\lim_{t\to\infty}\pr\left(\log\left(\int^{x/\phi(1/t)}_{0}\exp\{\alpha\xi_{s}\}\mathrm{d}s\right)>t\right)\\
&=\lim_{l\to\infty}\pr\left(\int^l_{0}\exp\left\{\alpha\xi_{s}\right\}\mathrm{d}s>\exp\{1/\varphi(x/l)\}\right)\\
&=\lim_{u\to\infty}\pr\left(x\left(\phi\left(\frac{1}{\log(u)}\right)\right)^{-1}>\tau\left(u\right)\right)\\
&=\lim_{u\to\infty}\p_{1}\left(x>\phi\left(\frac{1}{\log(u)}\right)\int^u_{0}X^{-\alpha}_{s}\mathrm{d}s\right),
\end{split}
\end{equation}
where the second equality is a consequence of the fact that $\widetilde{\xi}$ is self-similar with index $1/\beta$. So, using the well known fact that $\widetilde{\xi}_{1}^{-\beta}$ follows a Mittag-Leffler law of parameter $\beta,$ it follows therefrom that (i') implies (ii). Now, to prove that if (ii) holds then (i') does, simply use the previous equalities read from right to left. So, it remains to prove the equivalence between (i) and (ii) in the case $\beta=0.$ In this case we replace the first two equalities in equation (\ref{eq:ML*}) by $$\pr(e(1)<x)=\pr(\alpha\widetilde{\xi}_{x}>1),$$ and simply repeat the arguments above.

Given that the Mittag-Leffler distribution is completely determined by its entire moments the fact that (iii) implies (ii) is a simple consequence of the method of moments. Now we will prove that (i) implies (iii). Let $n\in\nr.$  To prove the convergence of the $n$-th moment of $\phi\left(\frac{1}{\log(t)}\right)\int^t_{0}X^{-\alpha}_{s}\mathrm{d}s$  to that of a multiple of a Mittag-Leffler r.v. we will use the following identity, for $x,c>0,$
\begin{equation}\label{eq:*}
\begin{split}
&\e_{x}\left(\left(c\int^t_{0}X^{-\alpha}_{s}\mathrm{d}s\right)^n\right)=\er\left(\left(c\tau(tx^{-\alpha})\right)^{n}\right)\\
&=c^n\int^\infty_{0}ny^{n-1}\pr(\tau(tx^{-\alpha})>y)\mathrm{d}y\\
&=\int^\infty_{0}ny^{n-1}\pr(\tau(tx^{-\alpha})>y/c)\mathrm{d}y\\
&=\int^\infty_{0}ny^{n-1}\pr\left(\log(tx^{-\alpha})>\alpha\xi_{y/c}+\log\int^{y/c}_{0}\exp\{-\alpha\xi_{s}\}\mathrm{d}s\right)\mathrm{d}y,
\end{split}
\end{equation}
where in the last equality we have used the time reversal property for L\'evy processes. By hypothesis, we know that for $y>0,$ $(\log(t))^{-1}\xi_{y/\phi\left(\frac{1}{\log(t)}\right)}\xrightarrow[t\to\infty]{\text{Law}}\widetilde{\xi}_{y},$ and therefore
\begin{equation*}
\begin{split}
&\pr\left(\log(tx^{-\alpha})>\alpha\xi_{y/\phi\left(\frac{1}{\log(t)}\right)}+\log\int^{y/\phi\left(\frac{1}{\log(t)}\right)}_{0}\exp\{-\alpha\xi_{s}\}\mathrm{d}s\right)\\&\sim\pr(1>\alpha\widetilde{\xi}_{y})\quad \text{as}\ t\to\infty.
\end{split}
\end{equation*}The claimed convergence will be then deduced from  the identity in (\ref{eq:*})  and the dominated convergence theorem. To justify the use of the dominated convergence theorem observe that for any $t,y>0$ such that $y>\phi(1/\log(t))$ we have  $$\log\int^{\frac{y}{\phi(1/\log(t))}}_{0}e^{-\alpha \xi_{s}}\mathrm{d}s\geq \log\int^{1}_{0}e^{-\alpha \xi_{s}}\mathrm{d}s\geq -\alpha\xi_{1},$$ and as a consequence $$\left\{\log(tx^{-\alpha})\geq \alpha\xi_{y/\phi(1/\log(t))}+\log\int^{\frac{y}{\phi(1/\log(t))}}_{0}e^{-\alpha \xi_{s}}\mathrm{d}s\right\}\subseteq \left\{\log(tx^{-\alpha})\geq \alpha\left(\xi_{y/\phi(1/\log(t))}-\xi_{1}\right)\right\}.$$ Using this, the fact that $\xi_{y/\phi(1/\log(t))}-\xi_{1}$ has the same law as $\xi_{\frac{y}{\phi(1/\log(t))}-1}$ and Markov's inequality it follows that the right most term in equation (\ref{eq:*}) is bounded by above by   
\begin{equation*}
\begin{split}
&\left(\phi(1/\log(t))\right)^n+\int^\infty_{\phi(1/\log(t))}ny^{n-1}\pr\left(\log(tx^{-\alpha})\geq \alpha\xi_{\frac{y}{\phi(1/\log(t))}-1}\right)\mathrm{d}y\\ 
&\leq \left(\phi(1/\log(t))\right)^n+\int^\infty_{\phi(1/\log(t))}ny^{n-1}\exp\left\{-\frac{\left(y-\phi(1/\log(t))\right)\phi\left(\alpha/\log\left(tx^{-\alpha}\right)\right)}{\phi(1/\log(t))}\right\}\mathrm{d}y\\
&\leq \left(\phi(1/\log(t))\right)^n+n2^{n-1}\frac{\left(\phi(1/\log(t))\right)^{n}}{\phi\left(\alpha/\log\left(tx^{-\alpha}\right)\right)}+2^{n-1}\Gamma(n+1)\left(\frac{\phi(1/\log(t))}{\phi(\alpha/\log(tx^{-\alpha}))}\right)^n.
\end{split}
\end{equation*}
The regular variation of $\phi$ implies  that the most right hand term in this equation is uniformly bounded for large $t.$

Therefore, we conclude that
\begin{equation*}
\begin{split}
\e_{x}\left(\left(\phi\left(\frac{1}{\log(t)}\right)\int^t_{0}X^{-\alpha}_{s}\mathrm{d}s\right)^n\right)&\xrightarrow[t\to\infty]{}\int^\infty_{0}ny^{n-1}\pr\left(1>\alpha\widetilde{\xi}_{y}\right)\mathrm{d}y\\
&=\begin{cases}\int^\infty_{0}ny^{n-1}\pr\left(e(1)>y\right)\mathrm{d}y,& \\
\int^\infty_{0}ny^{n-1}\pr\left(1>\alpha y^{1/\beta}\widetilde{\xi}_{1}\right)\mathrm{d}y,& \end{cases}\\
&=\begin{cases}n!,& \text{if}\ \beta=0,\\
\er\left(\left(\alpha^{-\beta}\widetilde{\xi}^{-\beta}_{1}\right)^n\right),& \text{if}\ \beta\in]0,1], \end{cases}
\end{split}
\end{equation*}
for any $x>0.$ We have so proved that (i) implies (iii) and thus finished the proof of Proposition \ref{DKprop}.

\begin{proof}[Proof of Corollary \ref{corDK}]
Observe that by Fubinni's theorem $\e_{1}\left(\int^t_{0}X^{-\alpha}_{s}\mathrm{d}s\right)=\int^t_{0}\e_{1}\left(X^{-\alpha}_{s}\right)\mathrm{d}s,$ and that the function $t\mapsto\e_{1}(X^{-\alpha}_{t})$ is non-increasing. So, by (iii) in Proposition \ref{DKprop} it follows that $$\int^t_{0}\e_{1}\left(X^{-\alpha}_{s}\right)\mathrm{d}s\sim \frac{1}{\alpha^{\beta}\Gamma(1+\beta)\phi\left(\frac{1}{\log(t)}\right)},\qquad t\to\infty.$$

Then, the monotone density theorem for regularly varying functions (Theorem 1.7.2 in \cite{BGT}) implies that $$\e_{1}\left(X^{-\alpha}_{t}\right)=o\left( \frac{1}{\alpha^{\beta}\Gamma(1+\beta)t\phi\left(\frac{1}{\log(t)}\right)}\right),\qquad t\to\infty.$$ Besides, given that $\e_{1}\left(X^{-\alpha}_{t}\right)=\er(e^{-tR_{\phi}}),$ for every $t\geq 0,$ we can apply Karamata's Tauberian Theorem (Theorem 1.7.1' in \cite{BGT}) to obtain the estimate $$\pr(R_{\phi}<s)=o\left(\frac{s}{\alpha^{\beta}\Gamma(1+\beta)\phi\left(\frac{1}{\log(1/s)}\right)}\right),\qquad s\to 0+.$$  Besides, applying Fubinni's theorem and making a change of variables of the form $u=sR_{\phi}/t$ we obtain the identity 
\begin{equation*}
\begin{split}
\int^t_{0}\e_{1}(X_{s}^{-\alpha})\mathrm{d}s=&\int^t_{0}\er(e^{-sR_{\phi}})\mathrm{d}s\\&=\er\left(\frac{t}{R_{\phi}}\int^{R_{\phi}}_{0}e^{-tu}\mathrm{d}u\right)\\
&=t\int^\infty_{0}\mathrm{d}ue^{-tu}\er\left(1_{\{R_{\phi}>u\}}\frac{1}{R_{\phi}}\right),\qquad t>0.
\end{split}
\end{equation*} So using Proposition \ref{DKprop} and Karamata's Tauberian Theorem we deduce that $$\er\left(1_{\{R_{\phi}>s\}}\frac{1}{R_{\phi}}\right)\sim\frac{1}{\alpha^\beta\Gamma(1+\beta)\phi(1/\log(1/s))},\qquad s\to 0+.$$

The proof of the second assertion follows from the fact that $I_{\phi}$ has the same law as $\alpha^{-1}R_{\theta}$ where $\theta(\lambda)=\lambda/\phi(\lambda),$ $\lambda>0,$ for a proof of this fact see the final Remark in \cite{BYfactorizations}.   
\end{proof}
\end{section}

\begin{section}{Proof of Theorem~\ref{th:2}}\label{proofth:2}
The proof of the first assertion in Theorem~\ref{th:2} uses a well known law of iterated logarithm for subordinators, see e.g. Chapter III in~\cite{bertoinbook}. Whilst, the second assertion in Theorem~\ref{th:2} is reminiscent of, and its proof is based on, a result for subordinators that appears in \cite{bertoin1995}. But to use those results we need three auxiliary Lemmas. The first of them is rather elementary.
\begin{lemma}\label{lemma:01}For every $c>0,$ and for every $f:\re^+\to\re^+,$ the a.s. equality of sets holds:
$$\mathcal{A}_1:=\{\xi_{\tau(s)}\leq c\log(s), \quad i.o.\ s\to\infty\}=\{\xi_{s}\leq c\log(C_s), \quad i.o.\ s\to\infty\}:=\mathcal{A}_2,$$
$$\mathcal{B}_1:=\{\xi_{\tau(s)}\geq f\left(\log(s)\right), \quad i.o.\ s\to\infty\}=\{\xi_{s}\geq f\left(\log(C_s)\right), \quad i.o.\ s\to\infty\}:=\mathcal{B}_2.$$
\end{lemma}
\begin{proof} We just prove the first equality, the second is proved
using a similar argument. Let $\omega\in\mathcal{A}_1$ and $(s_n,
n\geq 1)$ be any increasing sequence of reals, we claim that there
exists a subsequence $(s_{n_k}, k\geq 1)$ such that
$\xi_{s_{n_k}}(\omega)\leq c\log(C_{s_{n_k}}(\omega))$ for all $k$ large
enough. In that case we will have that $\omega\in\mathcal{A}_2$
owing to $(s_n, n\geq 1)$ is an arbitrary sequence. Indeed, let
$r_n:=C_{s_n}(\omega),$ $n\geq 1$ given that
$\omega\in\mathcal{A}_1$ we know that there exists a subsequence
$(r_{n_k}, k\geq 1)$ such that
$$\xi_{\tau(r_{n_k})}(\omega)\leq c\log(r_{n_k}),$$
for every $k$ large enough. It follows that the subsequence
$s_{n_k}=\tau(r_{n_k})(\omega)$ does the required work. Now the
inclusion $\mathcal{A}_2\subseteq\mathcal{A}_1$ is proved via the
same argument.
\end{proof}
\begin{lemma}\label{lemma:02}Under the assumptions of Theorem~\ref{th:2} we have the following estimates of the functional $\log\left(C_{t}\right)$ as $t \to \infty,$
\begin{equation}\label{eq:liminf}\liminf_{t\to\infty}\frac{\log\left(C_t\right)}{g(t)}=\alpha\beta(1-\beta)^{(1-\beta)/\beta}:=\alpha c_{\beta},\qquad
\pr\text{-a.s.}
\end{equation}
and
\begin{equation}\label{eq:limsup}\limsup_{t\to\infty}\frac{\log\left(C_t\right)}{\xi_{t}}=\alpha,\qquad
a.s.\end{equation}
\end{lemma}
\begin{proof}
Observe that
$$\log\left(C_t\right)\leq \log(t)+\alpha\xi_t,\qquad \forall t\geq 0,$$
so
$$\liminf_{t\to\infty}\frac{\log\left(C_t\right)}{g(t)}\leq \liminf_{t\to\infty}\left(\frac{\log(t)}{g(t)}+\frac{\alpha\xi_t}{g(t)}\right):=\alpha c_{\beta},\qquad \pr-\text{a.s.}$$
because $g$ is a function that is regularly varying at infinity with
an index $0<1/\beta.$ For every
$\omega\in\mathcal{B}:=\{\liminf_{t\to\infty}\frac{\xi_t}{g(t)}=c_{\beta}\}$
 and every $\epsilon>0$ there exists a $t(\epsilon,\omega)$ such
 that $$\xi_s(\omega)\geq (1-\epsilon)c_{\beta}g(s),\qquad s\geq
 t(\epsilon,\omega).$$ Therefore, $$\int^{t}_0\exp\{\alpha\xi_s\}\mathrm{d}s\geq\int^{t}_{t(\epsilon,\omega)}\exp\{(1-\epsilon)\alpha c_{\beta}g(s)\}\mathrm{d}s,\qquad \forall t\geq t(\epsilon,\omega),$$
and by Theorem 4.2.10 in \cite{BGT} we can ensure that
$$\lim_{t\to\infty}\frac{\log\left(\int^{t}_{t(\epsilon,\omega)}\exp\{(1-\epsilon)\alpha c_{\beta}g(s)\}\mathrm{d}s\right)}{(1-\epsilon)\alpha c_{\beta}g(t)}=1.$$
This implies that for every $\omega\in\mathcal{B}$ and $\epsilon>0$
$$\liminf_{t\to\infty}\frac{\log\left(C_t(\omega)\right)}{g(t)}\geq
(1-\epsilon)\alpha c_{\beta}.$$ Thus, by making $\epsilon\to 0+$ we obtain that for
every $\omega \in \mathcal{B}$
$$\liminf_{t\to\infty}\frac{\log\left(C_t(\omega)\right)}{g(t)}=\alpha c_{\beta},$$
which finish the proof of the first claim because $\pr\left(\mathcal{B}\right)=1.$

We will now proof the second claim. Indeed, as before we have that
$$\limsup\frac{\log\left(C_t\right)}{\xi_{t}}\leq \limsup_{t\to\infty}\frac{\log(t)+\alpha\xi_{t}}{\xi_{t}}=1,\qquad \text{a.s.},$$ owing to the fact
$$\lim_{t\to\infty}\frac{\xi_{t}}{t}=\er(\xi_1)=\infty,\qquad a.s.$$ Besides, it is
easy to verify that for every $\omega\in B$
$$\alpha c_{\beta}=\liminf\frac{\log(C_t)(\omega)}{g(t)}\leq\left[\liminf_{t\to\infty}\frac{\xi_{t}(\omega)}{g(t)}\right]\left[\limsup_{t\to\infty}\frac{\log\left(C_t(\omega)\right)}{\xi_{t}}\right],$$ and therefore that
$$\alpha\leq \limsup_{t\to\infty}\frac{\log\left(C_t\right)}{\xi_{t}},\qquad \text{a.
s.}$$ This finishes the proof of the a.s. estimate in
equation~(\ref{eq:limsup}).
\end{proof}
Using the Lemma~\ref{lemma:01} and the estimate~(\ref{eq:limsup}) the first assertion in Theorem~\ref{th:2} is straightforward. To prove the second assertion in Theorem~\ref{th:2} we will furthermore need the
following technical result.
\begin{lemma}\label{lemma:finiteness}
Under the assumptions of (ii) in Theorem~\ref{th:2} for any increasing function $f$ with positive increase we have that
\begin{equation}\label{eq:finiteness}
\int^\infty\phi\left(1/f(g(t))\right)\mathrm{d}t<\infty \Longleftrightarrow
\int^\infty\phi\left(1/f(cg(t))\right)\mathrm{d}t<\infty,\quad c>0.
\end{equation}
\end{lemma}
\begin{proof}
Our argument is based on the fact that $\phi$ and $g$ are functions
of regular variation at $0,$ and $\infty,$ respectively, with index $\beta$ and
$1/\beta,$ respectively, and on the fact that $f$ has positive
increase. Under the latter assumption we can assume that there is a
constant constant $M>0$ such that
$M<\liminf_{s\to\infty}\frac{f(s)}{f(2s)}.$ Thus for all $t, s$
large enough we have the following estimates for $g$ and $\phi.$
$$\frac{1}{2}\leq\frac{g(tc^\beta)}{cg(t)}\leq 2,\qquad
\frac{1}{2}\leq\frac{\phi\left(M/s\right)}{M^{\beta}\phi\left(1/s\right)}\leq
2.$$  Assume that the integral in the left side of the equation
(\ref{eq:finiteness}) is finite. It implies that the integral
$\int^\infty\phi\left(1/f(g(c^{\beta}t))\right)\mathrm{d}t<\infty,$
and  so that
\begin{equation*}
\begin{split}
\infty&>\int^\infty\phi\left(1/f(g(c^{\beta}t))\right)\mathrm{d}t\\
&\geq\int^\infty\phi\left(1/f(2cg(t))\right)\mathrm{d}t\\
&\geq\int^\infty\phi\left(\frac{f(cg(t))}{f(2cg(t))}\frac{1}{f(cg(t))}\right)\mathrm{d}t\\
&\geq\int^\infty\phi\left(M\frac{1}{f(cg(t))}\right)\mathrm{d}t\\
&\geq M^{\beta}\int^\infty\frac{\phi\left(M\frac{1}{f(cg(t))}\right)}{M^{\beta}\phi\left(\frac{1}{f(cg(t))}\right)}\phi\left(\frac{1}{f(cg(t))}\right)\mathrm{d}t\\
&\geq\frac{M^{\beta}}{2}\int^\infty\phi\left(\frac{1}{f(cg(t))}\right)\mathrm{d}t,
\end{split}
\end{equation*}
 where to get the second inequality we used that $f$ and $\phi$ are increasing
and the estimate of $g,$ in the fourth we used the fact that $f$ has
positive increase and in the sixth inequality we used the estimate
of $\phi.$ To prove that if the integral on the left side of
equation~(\ref{eq:finiteness}) is not finite then that the one in the right
is not finite either, we use that
$\limsup_{s\to\infty}\frac{f(s)}{f(s/2)}<M^{-1},$  and the estimates
provided above for $g$ and $\phi,$ respectively. We omit the
details.
\end{proof}

Now we have all the elements to prove the second claim of Theorem~\ref{th:2}.
\begin{proof}[Proof of Theorem~\ref{th:2}.b]
The proof of this result is based on Lemma~4 in~\cite{bertoin1995}
concerning the rate of growth of subordinators when the Laplace
exponent is regularly varying at $0$. Let $f$ be a function such
that the hypothesis in (b) in Theorem~\ref{th:2} is satisfied and
the integral in (\ref{inttest}) is finite. A consequence of
Lemma~\ref{lemma:finiteness} is that
$$\int^\infty\phi\left(1/f(\alpha c_{\beta}
g(t))\right)\mathrm{d}t<\infty.$$ On the one hand, according to the
Lemma~4~in~\cite{bertoin1995} we have that
\begin{equation*}
\limsup_{t\to\infty}\frac{\xi_t}{f(\alpha c_{\beta}g(t))}=0,\qquad \pr-\text{a.s.}
\end{equation*}Let $\Omega_1$ be the set of paths for which the
latter estimate and the one in~(\ref{eq:liminf}) hold. It is clear
that $\pr\left(\Omega_1\right)=1.$ On the other hand, for every
$\omega\in\Omega_1$ there exists a $t_0(\omega,1/2)$ such that
$$\alpha c_{\beta}g(s)/2\leq \log\left(C_s(\omega)\right),\qquad \forall s\geq
t_{0}(\omega,1/2),$$ with $c_{\beta}$ as in the proof of Lemma~\ref{lemma:02}. Which together with the fact $\limsup_{t\to\infty}\frac{f(t)}{f(t/2)}<\infty$ implies that for $\omega\in\Omega_1,$
$$\limsup_{s\to\infty}\frac{\xi_s(\omega)}{f\left(\log\left(C_s(\omega)\right)\right)}\leq
\limsup_{s\to\infty}\frac{\xi_s}{f(\alpha c_{\beta}g(s))}\frac{f(\alpha c_{\beta}g(s))}{f\left(\alpha c_{\beta}g(s)/2\right)}=0.$$
In this way we have proved that
\begin{equation}\label{eq:lim1}
\limsup_{s\to\infty}\frac{\xi_s}{f\left(\log\left(C_s\right)\right)}=0,\qquad\pr-\text{a.s.}\end{equation}
Now, let $f$ be a function such that the integral in
(\ref{inttest}) is not finite. As before applying the Lemma~\ref{lemma:finiteness} and the integral test
in Lemma~4 in~\cite{bertoin1995} we can ensure that
\begin{equation*}
\limsup_{t\to\infty}\frac{\xi_t}{f(\alpha c_{\beta}g(t))}=\infty,\qquad \pr-\text{a.s.}
\end{equation*} Denote by $\Omega_2$ the set of paths for which the
latter estimate and the one in~(\ref{eq:liminf}) hold. Let $(s_n,
n\in\nr)$ be a sequence of positive real numbers such that
$s_n\to\infty$ as $n\to\infty.$ For $\omega\in\Omega_2$ there exists
a subsequence $(s_{n_k}, k\in\nr)$ such that
\begin{equation*}
\lim_{k\to\infty}\frac{\xi_{s_{n_k}}(\omega)}{f(\alpha c_{\beta}g(s_{n_k}))}=\infty.
\end{equation*}
Furthermore, there exists a subsequence of $(s_{n_k}, k\in\nr),$ say
$(\widetilde{s}_{n_k}, k\in\nr),$ for which
$$\log\left(C_{\widetilde{s}_{n_k}}\right)\leq 2\alpha c_{\beta}g(\widetilde{s}_{n_k}),\qquad \forall k\in\nr.$$
The former and latter assertions imply that for $\omega\in\Omega_2$
\begin{equation*}
\lim_{k\to\infty}\frac{\xi_{\widetilde{s}_{n_k}}(\omega)}{f\left(\log\left(C_{\widetilde{s}_{n_{k}}}(\omega)\right)\right)}\geq
\lim_{k\to\infty}\frac{\xi_{\widetilde{s}_{n_k}}(\omega)}{f(\alpha c_{\beta}g(\widetilde{s}_{n_k}))}\frac{f(\alpha c_{\beta}g(\widetilde{s}_{n_k}))}{f\left(2\alpha c_{\beta}g(\widetilde{s}_{n_{k}})\right)}=\infty.
\end{equation*}
We deduce therefrom that
\begin{equation}\label{eq:lim2}\limsup_{t\to\infty}\frac{\xi_t}{f(\log\left(C_t\right))}=\infty,\qquad \pr\text{-a.s.}\end{equation}
Therefore, using the estimates in (\ref{eq:lim1}) and (\ref{eq:lim2}) together with Lemma \ref{lemma:01} we conclude the proof. 
\end{proof}
\end{section}
\begin{section}{On regularly varying splitting measures}\label{sect:frag}
To the best of our knowledge in the literature about self-similar fragmentation theory there is no example of self-similar fragmentation process whose dislocation measure is such that the hypothesis about the function $\Phi$ in Corollary \ref{cor:frag} is satisfied. In this section we will extend a model studied by Brennan and Durrett \cite{Brennandurrett2,Brennandurrett1} to provide an example of such a fragmentation process. Next, we will provide a necessary condition for a dislocation measure to be such that the hypothesis of Corollary \ref{cor:frag} is satisfied.

\begin{example}\label{brennandurret}
In ~\cite{Brennandurrett2,Brennandurrett1} Brennan \& Durrett studied a model that represents the evolution of a particle system in which a particle of size $x$ waits an exponential time of parameter $x^{\alpha},$ for some $\alpha>0,$ and then undergoes a binary splitting into a left particle of size $Ux$ and a right particle of size $(1-U)x.$ It is assumed that $U$ is a random variable that takes values in $[0,1],$ with a fixed distribution and whose law is independent of the past of the system. Assume that the particle system starts with a sole particle of size $1$ and that we observe the size of the left-most particle and write $l_{t}$ for its length at time $t\geq 0.$ It is known that the process $X:=\{X_{t}=1/l_{t}, t\geq 0\}$ is an increasing self-similar Markov process with self-similarity index $1/\alpha,$ starting at $1,$ see e.g. \cite{Brennandurrett2,Brennandurrett1} or \cite{BC2002}.  It follows from the construction that the subordinator $\xi$ associated to $X$ via Lamperti's transformation is a compound Poisson process with L\'evy measure the distribution of $-\log(U).$ That is, the Laplace exponent of $\xi$ has the form $$\phi(\lambda)=\e\left(1-U^{\lambda}\right),\lambda \geq 0.$$ In this case the Laplace exponent $\phi$ is regularly varying at zero with an index $\beta\in]0,1[$ if and only if $x\mapsto \p(-\log(U)>x)$ is regularly varying at infinity with index $-\beta.$ In particular the mean of $-\log(U)$ is not finite. If so, we have that $$\frac{-\log(l_{t})}{\log(t)}\xrightarrow[t\to\infty]{Law}V+\frac{1}{\alpha},$$ where $V$ is a r.v. whose law is described in Theorem \ref{th:1}. Observe that the limit law depends only on the index of self-similarity  and that one of regular variation of the right tail of $-\log(U)$ and not directly on the path of the underlying L\'evy process. Whilst in the case where $\e\left(-\log(U)\right)<\infty,$ it has been proved in \cite{Brennandurrett2,Brennandurrett1} and \cite{BC2002} that $l_{t}$ decreases as a power function of order $-\alpha$, and the weak limit of $t^{\alpha}l_{t}$ as $t\to\infty$ is $1/Z,$ where $Z$ is the r.v. whose law is described in (\ref{eq:weakconvergence}) and $(\ref{eq:limitlaw}
);$ so the limit law depends on the whole trajectory of the underlying subordinator. Besides, the first part of Theorem \ref{th:2} implies that $$\limsup_{t\to\infty}\frac{\log(l_{t})}{\log(t)}=-1/\alpha,\qquad \text{a.s.}$$ The $\liminf$ can be studied using the second part of Theorem \ref{th:2}. Furthermore, the  results in Corollary \ref{cor:frag} establish the convergence in probability of the empirical measure $\{\rho_{t}, t\geq 0\}$ associated to the fragmentation process that arises in this model.  
\end{example}

It is known, see \cite{bertoin2002} equation (8), that in general the splitting measure, say $\nu,$  of a self-similar fragmentation process is related to the L\'evy measure, say $\Pi,$ of the subordinator associated via Lamperti's transformation to the process of the tagged fragment, through the formula $$\Pi]x,\infty[=\int_{\mathcal{S}^{\downarrow}}\left(\sum_i s_i 1_{\{s_i<\exp(-x)\}}\right) \nu(\mathrm{d}s),\qquad x>0.$$ So the hypothesis of Corollary \ref{cor:frag} is satisfied with an index $\beta\in]0,1[$ whenever $\nu$ is such that 
\begin{itemize}
\item the function $x\mapsto \int_{\mathcal{S}^{\downarrow}}\left(\sum_i s_i 1_{\{s_i<\exp(-x)\}}\right) \nu(\mathrm{d}s),$ $x>0,$ is regularly varying at infinity with an index $-\beta$. 
\end{itemize}
In the particular case where $\nu$ is binary, that is when $\nu\{s\in\mathcal{S}^{\downarrow}: s_{3}>0\}=0,$ the latter condition is equivalent to the condition 
\begin{itemize}
\item
the function $x\mapsto \int_0^{\exp(-x)} y \nu(s_2 \in \mathrm{d}y)=\int^1_{1-\exp(-x)} (1-z) \nu(s_1 \in \mathrm{d}z),$ $x>0,$ is regularly varying at infinity with an index $-\beta,$
\end{itemize} given that in this case $s_1$ is always $\geq 1/2,$ and $\nu\{s_{1}+s_{2}\neq 1\}=0,$ by hypothesis.
\end{section}

\begin{section}{Final comments}\label{FC}
Lamperti's transformation tells us that under $\p_{1}$ the process $(\int^t_{0}X^{-\alpha}_{s}\mathrm{d}s, \log(X(t)), t\geq 0)$ has the same law as $(\tau(t), \xi_{\tau(t)}), t\geq 0)$ under $\pr.$ So, our results can be viewed as a study of how the time change $\tau$ modifies the asymptotic behaviour of the subordinator $\xi.$ Thus, it may be interesting to compare our results with those known for subordinators in the case where the associated Laplace exponent is regularly varying at $0$. 

On the one hand, we used before that the regular variation of the Laplace exponent $\phi$ at $0$ with an index $\beta\in]0,1],$ is equivalent to the convergence in distribution of $\varphi(1/t)\xi_{t}$ as $t\to \infty,$ to a real valued r.v., with $\varphi$ the right-continuous inverse of $\phi.$ On the other hand, Theorem \ref{th:1} tells us that   the former is equivalent to  the convergence in distribution of $\xi_{\tau(t)}/\log(t),$ as $t\to\infty,$ to a real valued random variable. Moreover, under the assumption of regular variation of $\phi$ with an index $\beta\in]0,1],$ we have that $\lim_{t\to\infty}\varphi(1/t)\log(t)=0.$ Thus we can conclude that the effect of $\tau(t)$ on $\xi$ is to slow down its rate of growth, which is rather normal given that $\tau(t)\leq t,$ for all $t\geq 0$, $\pr$-a.s.  Theorem \ref{th:1} tells us the exact rate of growth of $\xi_{\tau},$ in the sense of weak convergence.  Furthermore, these facts suggest that $\varphi(1/\tau(t))$ and $\log(t)$ should have the same order, which is confirmed by Proposition \ref {DKprop}. Indeed, using the regular variation of $\varphi$ and the estimate in (ii) in Proposition \ref {DKprop} we deduce the following estimates in distribution $$\varphi(1/\tau(t))\log(t)\sim \varphi(1/\tau(t))/\varphi(\phi(1/\log(t))) \sim (\alpha^{-\beta}W)^{-1/\beta}, \quad \text{as} \ t\to\infty,$$ where $W$ follows a Mittag-Leffler law of parameter $\beta.$ Observe also that if $\beta\in]0,1[,$ $\tau(t)$ bears the same asymptotic behaviour  as the first passage time for $e^{\alpha\xi},$ above $t,$ $L_{\log(t)/\alpha}=\inf\{s\geq 0, e^{\alpha \xi_{s}}>t\}.$ Indeed, it is known that under the present assumptions the process $\{t\xi_{u/\phi(1/t)}, u\geq 0\}$ converges, in Skorohod's topology, as $t\to\infty,$  towards a stable subordinator of parameter $\beta,$ say $\{\widetilde{\xi}_{t}, t\geq 0\}.$ This implies that $\phi(1/s)L_{s}$ converges weakly to the first passage time above the level $1$ for $\widetilde{\xi},$ and the latter follows a Mittag-Leffler law of parameter $\beta\in]0,1[.$  Which plainly justifies our assertion owing to Proposition \ref{DKprop}  and the fact that $\phi(1/\log(t))L_{\log(t)/\alpha}$ converges weakly towards a r.v. $\alpha^{-\beta}\widetilde{W},$ where $\widetilde{W}$ follows a Mittag-Leffler law of parameter $\beta.$ 

Besides, we can obtain further information about the rate of growth of $\xi$ when evaluated in stopping times of the form $\tau.$ It is known that if the $\phi$ is regularly varying with an index $\beta\in]0,1[,$ then $$\liminf_{t\to\infty}\frac{\xi_{t}}{g(t)}=\beta(1-\beta)^{(1-\beta)/\beta},\qquad \pr-\text{a.s.},$$ where the function $g$ is defined in Theorem \ref{th:2}. While the just cited Theorem states that $$\liminf_{t\to\infty}\frac{\xi_{\tau(t)}}{\log(t)}=\frac{1}{\alpha},\qquad \pr-\text{a.s.}$$ These together with the fact that $\lim_{t\to\infty}\frac{\log(t)}{g(t)}=0,$ confirms that the rate of growth of $\xi_{\tau(\cdot)}$ is slower than that of $\xi$, but this time using a.s. convergence. The long time behaviour of $\log(t)/g(\tau(t))$ is studied in the Proof of Theorem \ref{th:2}. The results on the upper envelop of $\xi$ and that of $\xi_{\tau}$ can be discussed in a similar way. We omit the details.
\end{section}

\gracias We would like to thank B. Haas for suggesting us the application of our results to self-similar fragmentations and insightful discussions about the topic. This research was funded by CONCyTEG (Council of Science and Technology of the state of Guanajuato, M\'exico) and partially by the project PAPIITT-IN120605, UNAM. 
\bibliographystyle{abbrv}
\bibliography{imateptbib}
\begin{flushleft}Instituto de Matem\'aticas\\ 
Universidad Nacional Aut\'onoma de M\'exico\\
Circuito Exterior, CU\\ 
04510 M\'exico, D.F.\\
M\'exico\\
\hspace{2cm}\\

Centro de Investigaci\'on en Matem\'aticas A.C.\\
Calle Jalisco s/n\\
Col. Valenciana\\
36240 Guanajuato, Gto.\\
M\'exico.\\
\end{flushleft}
\end{document}